\documentclass[reqno,b5paper]{amsart}
\usepackage{latexsym,graphicx,verbatim,color,cleveref,enumerate,fontenc,setspace,cite}
\usepackage{caption,amsthm,mathrsfs}
\usepackage{amsmath,amssymb}
\usepackage{subcaption}
\usepackage{enumerate}
\usepackage[mathscr]{eucal}
\newtheorem{theorem}{Theorem}[section]
\newtheorem{lemma}[theorem]{Lemma}
\newtheorem{proposizione}[theorem]{Proposition}
\newtheorem{corollario}[theorem]{Corollary}
\newtheorem{osservazione}[theorem]{Remark}
\newtheorem{definition}[theorem]{Definition}
\newtheorem{controes}[theorem]{Example}
\newtheorem{rem}[theorem]{Remark}
\newtheorem{as}[theorem]{Assumption}
\newtheorem{ass}[theorem]{Assumptions}

\newcommand{\iE}{\mbox{$\displaystyle{\int^*_E}$}}

\newcommand{\iAm}{\mbox{$\displaystyle{\int^*_{A_m}}$}}
\newcommand{\iS}{\mbox{$\displaystyle{\int^*_{S}}$}}

\newcommand{\iO}{\mbox{$\displaystyle{\int^*_{X}}$}}

\newcommand{\vuoto}{{\rm \O}}
\newcommand{\enne}{\mathbb{N}}
\newcommand{\erre}{\mathbb{R}}

\newcommand{\As}{$\mathcal A$~}

\newcommand{\Es}{\mbox{$\mathcal E$}}
\newcommand{\WEs}{\mbox{$W(\Es)$}}
\newcommand{\CEs}{\mbox{$C(\Es)$}}
\newcommand{\LCEs}{\mbox{$LC(\Es)$}}

\newcommand{\mint}{\displaystyle{\int^*}}
\newcommand{\mintX}{\displaystyle{\int_X^*}}
\newcommand{\mintA}{\displaystyle{\int_A^*}}
\newcommand{\mintS}{\displaystyle{\int_S^*}}
\newcommand{\mintB}{\displaystyle{\int_B^*}}
\newcommand{\mintEi}{\displaystyle{\int_{E_i}^*}}
\newcommand{\mintSEi}{\displaystyle{\int_{S\cap E_i}^*}}
\newcommand{\mintAEi}{\displaystyle{\int_{A\cap E_i}^*}}

\newcommand{\sommar}{\sum_{i=1}^r}

\newcommand{\uno}{{\bf 1}}
\numberwithin{equation}{section}
\newcommand{\acr}{\newline\indent}
\begin{document}
\title[The Choquet integral with respect to ...]{\bf  The Choquet integral with respect to fuzzy measures and applications}
\author[A. R. Sambucini]{\bf Anna Rita Sambucini}

\address{\llap{*\,}Department of Mathematics and Computer Sciences\acr
1, Via Vanvitelli  \acr 06123-I  Perugia \acr (ITALY)}
\email{anna.sambucini@unipg.it}
\thanks{This work was supported by   University of Perugia - Department of Mathematics and Computer Sciences - Grant Nr 2010.011.0403  and by the Grant   N.  U UFMBAZ2017/0000326 of GNAMPA - INDAM (Italy).\acr
Anna Rita Sambucini orcid ID: 0000-0003-0161-8729
}


\begin{abstract}
Fuzzy measures and Choquet asymmetric integral are considered here. 
As an application to economics some Core-Walras results are given.
\end{abstract}
\maketitle

\section*{Introduction}\label{introduzione}
Fuzzy measures  are used in economics, probability, theory of control and have been investigated by
many authors \cite{cdl,MuSu,D,pap2,mesiar,JKK,pap1,HS,JK,SC,bs1996,bdpm}.
The aim of the present paper is  to search  equilibria for pure exchange economies 
in  finite dimensional commodity spaces, but with  a more general structure for the set of agents $X$. 
After the framework proposed by Aumann, to represent a perfect competition
assuming that the space of agents is a non-atomic measure space, several extensions have appeared in the literature, see for example \cite{AR,Hilde}.
Some authors have considered finitely additive models, particularly in the coalitional sense (\cite{AR,B,BG,dms,BDG2009,BDG2010,V,BG2013,bgp2014}).
More recently finitely additive settings are examined in  \cite{MS0,MS1,MS2, MS3} and finitely additive economies with infinite dimensional commodity space have been investigated from the individualistic point of view  (\cite{AM1,AM4,M,M1}).
\acr
Here  the set of agents $X$ is equipped with a fuzzy measure $\mu$ representing the economic  weight of the  coalitions in the market, to compensate in some way  the loss of additivity  some further conditions are assumed.
Almost simple economies are considered; this  assumption is inspired by the models proposed in \cite{GCHB,BG,AM4,M}.
More precisely  the whole population $X $ can be decomposed into finitely many coalitions, and 
the agents within the
same coalition have both the same initial endowment and the same criteria of preference.
\acr
In order to do this the Choquet integral is considered since it permits  to
integrate a function with respect to a non-additive measure. Moreover  the Choquet
integral has been applied successfully in decision making, in artificial
intelligence and economics, see for example \cite{cmmm,bmv,gh,JKK,LIMP,mesiar,MuSu,mm,nt,t}, and it can be applied also in vector lattices (\cite{bs9596,bs1996,bs97,bs00}). \acr
The organization of the paper is as follows:  in Section \ref{choquet}  the properties of the monotone
integral are recalled and  new results on it are given while  in Section \ref{walras}  the main economic concepts to a fuzzy model are readapted and   competitive equilibria are compared
 with the core of the economy:   competitive equilibria will
belong to the core, for the converse it is proven that some particular allocations in the core are in fact competitive equilibria.
A particular model where the initial endowment turns out to be an equilibrium is also provided.
\acr
Finally  some particular cases are investigated, in which preferences are represented with concave,
 increasing functions:  characterizations of the core are given, and in some situations perfect equivalence Core-Walras is obtained. 
\section{Preliminaries and definitions}\label{preliminari}

\par \noindent
Let $(X, \mathcal{A})$ be a measurable space, where $\mathcal{A}$ is a $\sigma$-algebra.
\begin{definition}\rm (Murofushi and Sugeno \cite{MuSu})
A {\it fuzzy measure}  on a measurable space $(X, \mathcal{A})$  is a set function 
$\mu: \mathcal{A} \rightarrow \mathbb{R}^0_+$ with the properties:
 $\mu(\emptyset) = 0$; \quad  $\mu (X) < +\infty$;
 if $A \subset B$, then 
$\mu(A) \leq \mu(B)$  (monotonicity).
\end{definition}
Moreover 
\begin{itemize}
\item
$\mu$ is {\it subadditive} if for every $A, B \in \mathcal{A}$, there holds
$\mu(A \cup B) \leq \mu(A) + \mu(B),$
\item 
$\mu$ is a {\it submodular}, if for every $A, B \in \mathcal{A}$, there holds
$\mu(A \cup B) + \mu(A \cap B) \leq \mu(A) + \mu(B).$ In the literature submodularity is called also concavity (see \cite{mm}).
Note that if $\mu$ is a submodular then $\mu$ is subadditive.
\item $\mu$ is {\it continuous from below} if $A_n \in {\mathcal A}, A_{n+1} \subset A_n$ for $n \in \enne$ implies that
$\lim_{n \rightarrow \infty} \mu(A_n) = \mu (\cup_n A_n)$.
\end{itemize}

In classical measure theory it is well-known that if $f,g$ are measurable with $f=g$ 
outside a set of zero $\mu$-measure, then their integrals are equal. 
This is not valid in general for fuzzy measures and for the Choquet integral, 
for an example see \cite{SC}. 
So   the following definitions  of {\em null sets} and {\em almost everywhere} are introduced. 
In any case if  $\mu$ is a  subadditive fuzzy measure then every measurable set of $\mu$-measure zero is a null set 
(\cite[Proposition 5.2]{MuSu}).

\begin{definition}\label{ae} \rm (\cite{MuSu})
A set $N \in $ \As is a {\it null set} 
 (or also {\em $\mu$-null set})
 if $\mu(A \cup N) = \mu(A)$ for every $A \in $ \As. 
Clearly, every null set $N$ has $\mu$-measure zero (it suffices to choose $A=\emptyset$).
Moreover, it is easy to see that the family of all $\mu$-null sets is an ideal in \As and coincides with the family of all sets $N$ such that $\mu(A\setminus N)=\mu(A)$ for all $A\in$ \As.
\end{definition}
\begin{definition}\rm
{\em Almost everywhere} concept is defined  using null sets, that is 
$f=g \quad \mu$-a.e. if there exists a null set $N$ such that $f(x) = g(x)$ for every $x \in N^c$.
\end{definition}
Often the following condition will be assumed:
\begin{as} \rm
 the ideal of $\mu$-zero sets is stable under countable unions, namely
 for every sequence $(A_k)_k$ such that $\mu(A_k) = 0$ for each $k \in \enne$, then also
$\mu(\bigcup _k A_k) = 0$ .
\end{as}

\noindent
If $\mu$ is also subadditive this assumption means that a countable union  of $\mu$-null sets  is again $\mu$-null.
If $\mu$ is continuous from below, then a countable union of null sets is also a null set, and this in turn implies the same property 
with $\mu$-zero sets. (\cite[Proposition 5.2 (4)]{MuSu}).
Recent results in such subject and for universal integrals are given in \cite{LIMP2}.\acr

\noindent
The following definition is recalled:
\begin{definition} \label{semiconvex} 
\rm (\cite{CM1978,CM1979})
$\mu$ is \it semiconvex \rm if
for every $E\in \mathcal{A} $ there exists $F\in \mathcal{A} $, $F\subset E$ such that $\displaystyle{\mu (F) =
\dfrac{1}{2}\mu (E)}$ and $\displaystyle{\mu (E \setminus F) = \dfrac{1}{2} \mu (E)}$;
\end{definition}

Again in \cite{CM1978}, an example is given of a submodular and semiconvex fuzzy measure which is not additive. 
This example also satisfies the condition that the $\mu$-zero sets form a $\sigma$-ideal. This example is reported and extended here, by the sake of completeness.

\begin{controes}\label{perachille}\rm
Let $X=[0,1]^2$, with the usual Borel $\sigma$-algebra $\mathcal{B}$. For each $B\in \mathcal{B}$, define
$$\mu(B):=\int_0^1 \gamma(B_y)dy,$$
where $B_y$ is the usual $y$-section of $B$ (and is measurable for almost all $y\in [0,1]$), $\gamma(A):=\sqrt{\lambda_1(A)}$ and 
$\lambda_1$ is the Lebesgue measure in $[0,1]$.
In order to prove that $\mu$ is a submodular, it is sufficient to show that $\gamma$ has this property. Indeed, let $A,A'$ be any two subsets of $[0,1]$. The following inequalities will be proved:
$$\gamma(A\cup A')+\gamma(A\cap A')\leq \gamma(A)+\gamma(A').$$
To this aim,  denote by $a, b,c$ respectively the quantities $\lambda_1(A)$, $\lambda_1(A')$, $\lambda_1(A\cap A')$, and observe that $c\leq a\wedge b$.
Then the previous inequality reduces to
$$\sqrt{a+b-c}+\sqrt{c}\leq \sqrt{a}+\sqrt{b},$$
i.e.
$$a+b-c+c+2\sqrt{c} \cdot\sqrt{a+b-c}\leq a+b+2\sqrt{ab}.$$
So, it is enough to prove  that $\sqrt{c} \cdot \sqrt{a+b-c}\leq \sqrt{ab},$ or, equivalently, $c(a+b-c)\leq ab$. Now, a simple analysis of the parabola $y=x(a+b-x)$, for $0\leq x\leq a\wedge b$, shows that its maximum value is achieved when $x=a\wedge b$: and in this case the maximum value is precisely $ab$. This concludes the proof of submodularity. 
\acr
Semiconvexity is an easy consequence of additivity of the Lebesgue measure: indeed, for every Borel subset $B\subset X$ the mapping 
$y\mapsto \int_0^y \gamma(B_t) dt$ is continuous and therefore achieves the value $\dfrac{1}{2}\mu(B)$ for a suitable element $t_0\in [0,1]$, hence $B':=B\cap ([0,1]\times [0,t_0])$ satisfies 
$\mu(B')=\dfrac{1}{2}\mu(B)$.
\acr
Finally, in order to prove that the set of $\mu$-zero sets form a $\sigma$-ideal it is enough to show  that $\mu(B)=0\Longleftrightarrow \lambda_2(B)=0$:
 indeed, if $\mu(B)=0$, then $\gamma(B_y)=0$ $\lambda_1$-a.e., and so $\lambda_1(B_y)=0$ $\lambda_1$-a.e.,
which in turn implies that $\lambda_2(B)=0$. Conversely, if $\lambda_2(B)=0$ then $\lambda_1(B_y)=0$ for almost all $y\in [0,1]$ and therefore $\gamma(B_y)=0$ $\lambda_1$-a.e., whence $\mu(B)=0$.
\end{controes}

In the same paper  the following properties of semiconvex submeasures have been established.
\begin{lemma}\label{CM2.1} 
\rm(\cite[ Lemma 2.1]{CM1978}) \em If $\mu :\mathcal{A} \rightarrow 
\mathbb{R}^0_+$
is a fuzzy,  finite
and semiconvex submeasure, then for every $A\in \mathcal{A}$ there exists a family of subsets of $A$,
$(A_t)_{t\in [0,1]}\subset \mathcal{A} $ such that
\begin{itemize}
\item[(\ref{CM2.1}.i)] $A_0 = \vuoto, \quad A_1 = A$;
\item[(\ref{CM2.1}.ii)] $\mu (A_t) = t\mu (A)$;
\item[(\ref{CM2.1}.iii)] for $t < t'$, there hold $A_t \subset A_{t'}$ and $\mu(A_{t'} \setminus A_t) = (t' -t) \mu(A)$.
\end{itemize}
In particular, for every  $A \in \mathcal{A}$, the range of $\mu$ restricted to $A$ is $[0,\mu(A)]$.

\end{lemma}

\begin{proposizione}\label{zero} 
For every $A,B \in \mathcal{A} $ with $A\subset B$, 
for every $t\in ]0,1[$,
and once $A_t\in \mathcal{A}$ has been chosen according to { Lemma {\rm \ref{CM2.1}} }
there exists $B_t\in \mathcal{A}$, $B_t\subset B$ with $B_t\cap A = A_t$ and $\mu (B_t) = t\mu(B)$.
\end{proposizione}
\begin{proof} 
 By { Lemma \ref{CM2.1}} there exist two families $(A_s)_{s\in[0,1]}$, $((B\setminus A)_s)_{s\in[0,1]}$
with the properties (\ref{CM2.1}.i), (\ref{CM2.1}.ii) and (\ref{CM2.1}.iii).
Choose $A_t$ in the first family, and consider the family
\[ {\mathcal B} = \{ A_t\cup (B\setminus A)_s, s\in[0,1]\} .\]
Then, as in \cite[Theorem 2.3]{CM1978}, $\mathcal B$ is arcwise connected with respect to the pseudometric on $d_{\mu} :\mathcal{A} \times  \mathcal{A} \rightarrow \mathbb{R}$ defined by:
\(d_{\mu} (E,F) = \mu(E \Delta F).\)
Moreover, $\mu:({\mathcal A}, d_{\mu}) \rightarrow \erre$ is uniformly continuous since, for every pair $A_1, A_2 \in {\mathcal A}$, it is:
\begin{eqnarray*}
\mu(A_1) - \mu(A_2) \leq \mu(A_1 \cap A_2) + \mu(A_1 \setminus A_2) - \mu(A_2) \leq \mu (A_1 \Delta A_2)
\end{eqnarray*}
and analogously for $\mu(A_2) - \mu(A_1)$.
Therefore $\mu(\mathcal{B})$ is the interval
\( \mu (\mathcal{B}) = [ \mu(A_t), \mu(A_t \cup (B\setminus A))]. \)
It only remains to prove that $t\mu(B)$ lies in this interval.
From $\mu(A) \leq \mu (B)$ it follows that $\mu(A_t)=t\mu(A)\leq t\mu(B)$.
It will be shown now  that $t\mu (B) \leq \mu (A_t \cup (B\setminus A))$.
Observe that, by construction
\( A_t \cup (B\setminus A) = B \setminus (A \setminus A_{t}), \)
and so, since by (\ref{CM2.1}.iii), $\mu (A \setminus A_{t}) = (1-t) \mu (A)$,
\begin{eqnarray*}
\mu (A_t \cup (B\setminus A)) &=& \mu(B \setminus (A \setminus A_{t})) \geq \mu (B) - (1-t) \mu (A)  
=   \\ &=& t\mu (B) + (1-t)[\mu(B) - \mu (A)] \geq  t\mu(B).
\end{eqnarray*}
\noindent
Hence there exists $K\in {\mathcal B}$ such that $\mu(K) =t\mu(B).$
\par \noindent
Setting $B_t = K$ we have $B_t = A_t \cup (B\setminus A)_s$ for suitable $s$, and so $B_t\cap A = A_t$.
\end{proof}
On this subject other results have been obtained in \cite{CM1992,bcm}.


\section{The Choquet integral}\label{choquet}
Let $(X, \mathcal{A}, \mu)$ be a subadditive fuzzy measure space. A 
 function $f:X \rightarrow \erre^+_0$  is said to be {\it measurable} if $\{x | f(x) > \alpha\}$ is measurable for every $\alpha > 0$.
\begin{definition} \rm
The  \it Choquet integral \rm of a measurable function $f$ is defined by 
\[ \iO f d\mu:= \int_0^{\infty} \mu( f > t) dt,\]
where the integral on the right-hand side is an ordinary one.
The function  $f \in L^1_C(\mu)$ if and only if $f$ is measurable and  
$\iO f d\mu < \infty$.\acr
In the case of a measurable function $f$ with values in $\mathbb{R}$, according to \cite{D}
 the {\it asymmetric integral} \rm  
is defined as
\begin{eqnarray}
\iO f d\mu=
\int_0^{\infty} \mu( f > t) dt +
\int _{-\infty}^0 [\mu(f > t) - \mu(X)]dt.
\end{eqnarray}
\end{definition}
For   every $A \in \mathcal{A}$  the {\it conjugate} of $\mu$  is defined by: $\overline{\mu}(A) := \mu(X) - \mu(A^c)$.
\acr

\noindent
The Choquet integral fulfils the following properties:
\begin{proposizione}\label{De} \rm  (\cite[Proposition 5.1, Exercise 9.3 and Theorem 6.3]{D})
\begin{itemize}
\item[(\ref{De}.i)] $\iO 1_A d\mu= \mu(A)$;
\item[(\ref{De}.ii)] $\iO c f d\mu= c \iO f d\mu$ for $c \geq 0$;
\item[(\ref{De}.iii)] if  $0 \leq f \leq g$ $\mu$-a.e. then  $\iO f d\mu \leq \iO g d\mu$  (\cite[Exercise 9.3(c)]{D});  
\item[(\ref{De}.iv)] $\iO (f + c) d\mu= \iO f d\mu+ c\mu(X)$ for every $c \in \mathbb{R}$; 
\item[(\ref{De}.v)] if $\mu$ is a submodular and 
$\mu (f > -\infty) = \mu (g > -\infty) = \mu(X),$
 then
\[\iO (f+g) d\mu\leq \iO f d\mu+ \iO g d\mu; \hskip.5cm
\mbox{(\cite[Theorem 6.3]{D}) }
\] (observe that if $f,g$ are non negative the $\mu$-essential inf-boundedness is satisfied; moreover this assumption can be dropped if $\mu$ is continuous from below.)
Conversely, if \[\iO (f+g) d\mu\leq \iO f d\mu+ \iO g d\mu
\]
holds for all measurable non negative mappings $f,g$, then $\mu$ turns out to be  submodular: it suffices to replace $f$ and $g$ with the functions $1_A$ and $1_B$ for arbitrary measurable sets  $A,B$.
\item[(\ref{De}.vi)]  
\( \displaystyle{\iO -f d\mu= - \iO f d\overline{\mu}}.\)
\end{itemize}
\end{proposizione}

\begin{definition}\label{indefinite} \rm
Given a non negative function $f\in L^1_{C}(\mu) $ let  $\mu _f $ be its indefinite Choquet integral, that is
\begin{eqnarray}\label{int-suE} 
\mu_f(E) = \iE f d\mu= \iO {\bf 1}_E f d\mu.
\end{eqnarray}
\end{definition}
It is clear then that $\mu_f$ is a fuzzy measure too.
Observe that for arbitrary $f$ the equality (\ref{int-suE}) fails to be true for the asymmetric integral: (see \cite[Chapter 11]{D}).

\begin{proposizione}\label{pmis} \rm
Let $f,g$ be two non negative functions in $ L^1_C (\mu)$; then:
\begin{itemize}
\item[\ref{pmis}.1)]  $\mu_f$ is a subadditive  fuzzy measure;
\item[\ref{pmis}.2)] if $\mu$ is also a submodular then $\mu_f$ is a submodular.
\item[\ref{pmis}.3)] if the family of $\mu$-zero sets is a $\sigma$-ideal then also $\mu_f$ has the same property.
\end{itemize}
\end{proposizione}
\begin{proof}
\begin{itemize}
\item[\ref{pmis}.1)] 
Clearly $\mu_f$ is obviously fuzzy. Moreover
\begin{eqnarray*}
\mu (\{x \in A \cup B : f(x) > t\}) \leq \mu (\{x \in A : f(x) > t\}) + \mu (\{x \in  B : f(x) > t\}) 
\end{eqnarray*}
and the assertion follows from both (\ref{int-suE}) and the additivity of the Lebesgue integral.
\item[\ref{pmis}.2)] 
 Analogously, for the submodular case we have:
\begin{eqnarray*}
\mu_f(A \cup B) &=& \int_{A \cup B}^* f d\mu = \int_0^{\infty} \mu (\{ x \in A \cup B : f(x) > t \}) dt \leq \\ &\leq&
\int_0^{\infty} \mu (\{ x \in A : f(x) > t \}) dt  + \int_0^{\infty} \mu (\{ x  \in  B : f(x) > t \}) dt  + \\ &-&
\int_0^{\infty} \mu (\{ x \in A \cap B : f(x) > t \}) dt;
\end{eqnarray*}
hence 
$\mu_f(A \cup B) + \mu_f(A \cap B) \leq \mu_f(A) + \mu_f(B).$
\item[\ref{pmis}.3)]
Now, assume that the $\mu$-zero sets form a $\sigma$-ideal, and let $(B_n)_n$ be any sequence of $\mu_f$-zero sets. Let  $B$ the union of all the sets $B_n$. Then, by properties of the Lebesgue integral, for each integer $n$ there exists a Borel subset $N_n$ of the halfline $[0,+\infty[$ 
with zero Lebesgue measure, such that $\mu(\{f>t\}\cap B_n)=0$ for all $t\notin N_n$. So, if $N=\bigcup_n N_n$, it is $\lambda(N)=0$ and 
$\mu(\{f>t\}\cap B_n)=0$  for all $t\notin N$ and every $n$. Since the ideal of $\mu$-zero sets is a $\sigma$-ideal then $\mu(\{f>t\}\cap B)=0$,
 for every $t\notin N$, and hence $\mu_f(B)=0$.
\end{itemize}
\end{proof}

Observe that ({\ref{pmis}.1}) continues to be true for the
 asymmetric integral of an arbitrary $f$ if $\mu$ is a submodular, while  ({\ref{pmis}.2}) fails to be true (\cite[Chapter 11]{D}).\acr

The indefinite Choquet integral  of non negative functions is absolutely continuous with respect to $\mu$ in a stronger form 
than \cite[Chapter 11]{D}: namely
\begin{proposizione}\label{ac}
Let $f: X \rightarrow \mathbb{R}_+$ be an integrable function. Then, for every $\varepsilon > 0$ there exists $\delta(\varepsilon) > 0$ such 
that for every $E \in \mathcal{A}$ with $\mu(E) < \delta(\varepsilon)$ then
\( \mu_f (E) < \varepsilon.\)
\end{proposizione}
\begin{proof} 
Since the Choquet integral of non negative scalar functions is an improper integral then,  for every $\varepsilon > 0$, there exists $a \in \mathbb{R}_+$ such that
$\int_a^{\infty} \mu(x: f (x) > t) dt \leq \dfrac{\varepsilon}{2}$.
So, for every $E \in \mathcal{A}$, by (\ref{int-suE}) it is
\begin{eqnarray*}
\mu_f (E) = \int_E^* f d\mu&=& \int_0^a \mu (x \in E : f (x) > t ) dt +
\int_a^{\infty} \mu (x \in E : f(x) > t) dt \leq \\&\leq&
 a \mu(E) + \frac{\varepsilon}{2}.
\end{eqnarray*}
So it is enough to choose $\delta(\varepsilon) = \varepsilon/2a$. 
\end{proof}
On this subject see also \cite{cdlds}. 
The next proposition permits to obtain the same results as in  \cite[Proposition 11.1]{D} but with different hypotheses:
 in the quoted proposition $\mu$ is asked to be monotone and continuos from below, here   the first 
condition will be strengthened
while the second will be weakened.
\begin{proposizione}\label{muf} 
If $\mu$ is a subadditive fuzzy measure whose zero sets form a $\sigma$-ideal
\it and if $f,g$ are two non-negative measurable functions
with $\mu_f \leq \mu_g$, then $f \leq g~ \mu$-a.e.
\end{proposizione}
\begin{proof}
Assume by contradiction $\mu (f > g)$ is positive. Since
$A = \{ f > g\} = \cup_k \{ f > g + 1/k\}= \cup_k A_k$,  
then there should exist $m$ for which
$\mu (A_m) > 0$ (otherwise, by the $\sigma$-ideal condition,
 $\mu (A)$ should be zero too).
Then
\[\mu_f (A_m) \geq \iAm (g + \dfrac{1}{m}) d\mu=
\mu_g(A_m) + \dfrac{1}{m} \mu(A_m) > \mu_g(A_m).\]
 \end{proof}

In  measure theory it is well-known that if $f,g$ are measurable then $f=g$ $\mu$-a.e. (in the classical sense), if and only if their integrals are equal. This is not valid  for the Choquet integral unless we assume
that the  ideal of $\mu$-zero sets is stable under countable unions:
\begin{corollario}\label{ugualeqo}
If $\mu$ is a subadditive fuzzy measure whose zero sets form a $\sigma$-ideal,
$f,g$ are two non negative and measurable functions
and $\mu_f (E) = \mu_g (E)$
for every $E \in {\mathcal A}$,  then 
 $f=g$ in $X$ $\mu$-a.e.
\end{corollario}
\begin{proof}
Since  $\mu_f (E) \leq \mu_g(E)$  and $\mu_f(E) \geq \mu_g(E)$ for every $E \in {\mathcal A}$ the assertion
follows from Proposition \ref{muf}.
\end{proof}

For simple functions, without subadditivity assumption on $\mu$, one has:
\begin{proposizione}\mbox{\rm \cite[ Chap. 5, pag. 63]{D}}\label{int-simple}
Let $f: X \rightarrow \mathbb{R}^0_+$ be defined by
$f (x) = \sum_{i=1}^r x_i 1_{A_i} (x),$ 
with $A_i \cap A_j = \vuoto$ if $i \neq j$. If ~
$x_i > x_{i+1},~ i=1,\ldots, r-1$, 
then
\begin{eqnarray}\label{sommemon}
\int_{X}^* f d\mu&=&
\sum_{i=1}^{r} (x_i - x_{i+1}) \mu(S_i)
\end{eqnarray}
where $S_i = \cup_{j \leq i} A_j$ and $x_{r+1} =0$.
\end{proposizione}

As a consequence of Proposition \ref{int-simple},  a Jensen-type inequality for the Choquet integral will be stated.
 This result is  still valid without subadditivity assumptions.
This inequality was already given by \cite{LIMP,PT, WANG} under suitable assumptions, however  here a more direct proof will be stated, relying essentially upon Proposition \ref{int-simple}. For further use,  only the formulation for {\em concave} functions will be given.

\begin{theorem}\label{Jensen} {\rm (Jensen inequality)}
Let $f:X\to \erre_+^0$ be any integrable function, and assume that $u:\erre_+^0\to \erre_+^0$ is any 
concave continuous monotone mapping such that $u(f)$ is integrable too. Then it is
$$u(\mu(X)^{-1}\mu_f (X))\geq \mu(X)^{-1}\mu_{ u(f) } (X).$$
\end{theorem} 
\begin{proof} 
Without loss of generality assume that $\mu(X)=1$.
As a first step, suppose that $f$ is simple, i.e. 
$f (x) = \sum_{i=1}^k x_i 1_{A_i} (x),$
assuming $x_i>x_{i+1}$ for all $i=1,...,k-1$.\acr
Then
$u(f(x))= \sum_{i=1}^k u(x_i) 1_{A_i} (x);$
since $u$ is monotone it is $u(x_i)\geq u(x_{i+1})$ for all $i$.
Then, using (\ref{sommemon}), one can write
\begin{eqnarray*}
 \mu_f(X) &=& x_k\mu(X)+(x_{k-1}-x_k)\mu(S_{k-1})+\ldots+(x_3-x_2)\mu(A_1\cup A_2)+
\\ &+&(x_2-x_1)\mu(A_1)+ x_1\mu(A_1)=\\
&=& x_k\big(\mu(X)-\mu(S_{k-1})\big)+x_{k-1}\mu(S_{k-1})+\ldots+(x_3-x_2)\mu(A_1\cup A_2)+
\\ &+& (x_2-x_1)\mu(A_1)+ x_1\mu(A_1)=\\
&=& x_k\big(\mu(X)-\mu(S_{k-1})\big)+x_{k-1}\big(\mu(S_{k-1})-\mu(S_{k-2})+
x_{k-2}\mu(S_{k-2})+\ldots \\ &+& (x_2-x_1)\mu(A_1)+ x_1\mu(A_1)=\\
&=& x_k\big(\mu(X)-\mu(S_{k-1})\big)+x_{k-1}\big(\mu(S_{k-1})-
\mu(S_{k-2})\big)+ \ldots+ \\ &+& x_2\big(\mu(A_1\cup A_2)-\mu(A_1)\big)+ x_1\mu(A_1).
\end{eqnarray*}
This formula shows that the integral of $f$ is a convex combination of the elements $x_k,x_{k-1},...,x_1$, 
obtained with the positive coefficients 
$$\mu(X)-\mu(S_{k-1}), \mu(S_{k-1})-\mu(S_{k-2}),\ldots,
\mu(A_1\cup A_2)-\mu(A_1), \mu(A_1),$$ whose sum is $\mu(X)=1$.
And clearly the same coefficients appear in the formula giving  $\mu_{u(f)} (X)$.
So, by concavity of $u$ it follows easily that
$u(\mu_f (X)) \geq \mu_{u(f)} (X)$.
and this proves the theorem for the case $f$ simple. \acr
If $f$ is not simple, but bounded, then $f$ can be uniformly approximated  with a sequence of simple 
functions $f_n$ (the usual {\em Lebesgue ladder} does the job). Then $u(f)$ is uniformly approximated 
by the sequence $u(f_n)$, since $u$ is locally uniformly continuous.
Then 
$$u(\mintX f d\mu)=\lim_n u(\mintX f_n d\mu)\geq  
\lim_n\mintX u(f_n) d\mu=\mintX u(f)d\mu.$$
Finally, suppose that $f$ is unbounded. In this case, for each integer $N$ define $f_N:=f\wedge N$, and observe that
$$\mintX fd\mu=\lim_N \mint f_Nd\mu.$$
Then 
$$u(\mintX f_N d\mu)\geq \mintX u(f_N)d\mu=\mintX u(f)\wedge u(N)d\mu.$$
The conclusion now follows observing that 
$$\lim_Nu(\mintX f_N d\mu)=u(\mintX f d\mu) \quad   \text{ and  that}\quad  \lim_N\mintX u(f)\wedge u(N)d\mu=\mintX u(f)d\mu.$$ 
\end{proof}

The Jensen inequality given here holds (in the opposite sense) also when concavity is replaced by
 convexity; its  proof is perfectly similar to this.

\subsection{The vector Choquet integral}
In $\erre ^n$ let  $\mathbb{R}_+^n$ be the positive orthant, and  $(\mathbb{R}_+^n)^{\circ}$ be its interior.
Also let $\leq $ be the usual order between numbers, and  \underline{$\ll $} be  the usual
partial order between vectors in $\erre ^n$, namely $ x \underline{\ll} y$  means that $x_i \leq y_i$ for every $i=1,2, \cdots, n$, while
$x \gg y$ means that $x_i > y_i$ for every $i$.

\begin{definition}\rm 
Given a vector measurable function $f = (f_1, \ldots f_n): X \rightarrow \mathbb{R}_+^n$   the monotone integral is considered componentwise,
 and the  notation $\mint f d\mu$ is used for the vector 
$$\mint fd\mu= \left( \mint f_1d\mu, \ldots ,\mint f_nd\mu\right).$$
Then $f\in L^1_{C}(\mu, \mathbb{R}_+^n)$ if each of its components is in $L^1_{C}(\mu)$.
\end{definition}

\begin{as}\rm
Suppose now that  $\mu$  is a fuzzy measure which satisfies 
the following condition:
      \begin{itemize}
          \item[(H.0)] there exists a partition 
$\{E_i, i=1 , \ldots, r \}$   of $X$, such that for every $A\in \mathcal{A}$
                   $$ \mu (A) = \sum_{i=1}^r \mu(A \cap E_i).$$ 
      \end{itemize}
 Assumption ({\bf H.0})  means that
$\mu$ can be cut in a  finitely additive way on $\{E_i, i=1 , \ldots, r \}$, namely
$ \mu (A \cap \cup_{j \leq k} E_j ) = \sum_{j=1}^k \mu(A \cap E_j), \text{ for every } A \in \mathcal{A}, ~
k \leq r;$
in fact
\[ \mu(A \cap (\cup_{j \leq k} E_j) ) = \sum_{i=1}^r \mu(A \cap (\cup_{j \leq k} E_j) \cap E_i)
= \sum_{j=1}^k \mu(A \cap E_j).
\]
\end{as}

\begin{rem}\rm 
Trivially, the measure $\mu$ defined in Example \ref{perachille} satisfies assumption $\bf{(H.0})$ with $n=1$, 
however it can be easily modified in
order to admit an arbitrary finite number of subsets of the type $E_i$: it will suffice to  paste together the squares 
$E_i:=[i,i+1[\times [0,1]$, each with a  copy of the measure $\mu$, and defining additively the measure of all 
sets that are unions of measurable subsets of  the $E_i's$.
\end{rem}

From now on  ({\bf H.0}) will be assumed and  denote by $\{E_i, i \leq r \}$ the finite decomposition of $X$ involved in it.
\begin{proposizione}\label{intadditivo}
If $g :X \rightarrow \mathbb{R}^n_+$ is in $L^1_C(\mu,\mathbb{R}^n_+)$, then, for every $A \in \mathcal{A}$,   
\begin{itemize}
\item[\ref{intadditivo}.1)]
$\mu_g(A) = \sommar \mu_g(A \cap E_i)$.
\item[\ref{intadditivo}.2)]
In particular, if  $g$ is of the form
\( g = \sommar c_i 1_{E_i}\)
then, for every $A \in \mathcal{A}$,  its Choquet integral is given by
$\mu_g(A) = \sommar c_i \mu (A \cap E_i).$
\item[\ref{intadditivo}.3] If $\mu$ is  submodular and $f :X \rightarrow \mathbb{R}^n_+$ is in $L^1_C(\mu,\mathbb{R}^n_+)$ then,  
 for every $A, B \in \mathcal{A}$ it is
\[ \mintX (f \uno_A + g \uno_B) d \mu \quad  \underline{\ll } \mintA  f d\mu + \mintB   g d\mu.\]
\end{itemize}
\end{proposizione}
\begin{proof}
Fix $A\in \mathcal{A}$: then
\begin{eqnarray*}
\mu_g(A) &=& \mintA  g d\mu = \int_0^{\infty}\mu(\{x\in A:g(x)>t\} )dt=
\\ &=&
 \int_0^{\infty}\sommar\mu(\{x\in A\cap E_i:g(x)>t\})dt= \\
&=& \sommar \int_0^{\infty}\mu(\{x\in A\cap E_i:g(x)>t\})dt=  \\ &=&
\sommar \mintAEi gd\mu = \sommar \mu_g(A\cap E_i).
\end{eqnarray*}
The second part of the assertion is trivial.
The last assertion follows directly from ({\ref{De}.v}) and equation (\ref{int-suE}).
\end{proof}

 So, assumption ({\bf H.0}) makes the integral of simple functions  on $(E_i)_{i \leq r}$ additive. Moreover
\begin{proposizione}\label{corollario}
If $g: X \rightarrow \mathbb{R}^n_+$ is Choquet integrable then, for every $p \in \mathbb{R}^n_+$,
\begin{itemize}
\item[\ref{corollario}.a)] $\mu_{p  \cdot g} (X) = \sum_{i=1}^r \mu_{p  \cdot g} (E_i),$
\item[\ref{corollario}.b)]
moreover, if $g$ is constant on each $E_i$ {\rm(}namely $g(x) =c_i$ on $E_i${\rm)}  then
$$\mu_{p  \cdot g} (X) =  \sum_{i=1}^r p \cdot \mu_g (E_i) = p \cdot \mu_g (X).$$
\end{itemize}
\end{proposizione}
\begin{proof}
\ref{corollario}.a) is a direct consequence of
\ref{intadditivo}.1), just replacing the vector function $g$ with the scalar map $p\cdot g$. The last statement follows
 readily from  the fact that on each $E_i$ the function $p \cdot g = p \cdot c_i$ is constant 
and so, by linearity of $p$ 
\begin{eqnarray*}
\mu_{p  \cdot g} (X) &=&  \mintX p  \cdot g d\mu = \sum_{i=1}^r  \mintEi   p \cdot g d\mu =
\sum_{i=1}^r  \mintEi   p \cdot c_i d\mu = \sum_{i=1}^r   p \cdot c_i \mu(E_i) = \\ &=&
p \cdot \sum_{i=1}^r  c_i \mu(E_i) = p \cdot \mintX g d\mu = p \cdot \mu_g (X).
\end{eqnarray*} 
\end{proof}
 The following lemma is a consequence ot the previous two propositions.
\begin{lemma}\label{eq-zero}
Let $f,g$ be two scalar integrable non-negative mappings, and assume that   and $f>g$ $\mu-a.e.$ on a set $S$ with $\mu(S)>0$. 
If $g$ is constant on each  $E_i$, then $\mu_f (S) > \mu_g(S)$.
\end{lemma}
\begin{proof} 
Let  $g:=\sommar c_i 1_{E_i}.$
Thanks to Propositions \ref{intadditivo} and \ref{corollario}, it is
$$\iS f d\mu=\sommar \mintSEi   f d\mu,$$
and 
$$\iS g d\mu=\sommar \mintSEi   g d\mu=\sommar  c_i \mu(S\cap E_i).$$
Now, for each $i$ by monotonicity it follows that
$$\mintSEi   f  d\mu\geq c_i \mu(S\cap E_i),$$
so, it will suffice to find an index $i$ such that the strict inequality holds. Since $\mu(S)>0$, there exists at least
 an index $i$ for which $\mu(S\cap E_i)>0$: this is the  requested index, in fact  
in the set $S\cap E_i$ the function $f$ is strictly larger than the constant $c_i$.
From the basic property  ({\ref{De}.iv}) it is
$$\mintSEi   (f -c_i)d\mu=\mintSEi   f d\mu-c_i\mu(E_i\cap S).$$
So,  it is enough to prove that the strictly positive function $f  - c_i$ (in the set $E_i\cap S$) has strictly positive integral.
But now, thanks to the properties of $\mu$-zero sets,
 there exists an integer $k$ such that the set $A_k:=\{a\in S\cap E_i: f  - c_i>\dfrac{1}{k}\}$ has positive measure, 
otherwise $\mu(E_i\cap S)=0$: then 
$$\mintSEi   (f-c_i)d\mu\geq \int_{A_k}^* (f -c_i)d\mu\geq d\frac{1}{k}\mu(A_k)>0$$
and this concludes the proof. 
 \end{proof} 

\section{Walrasian equilibria and core of an economy}\label{walras}

\noindent
The following economic model is introduced: 
\begin{as} \rm 
A \it pure exchange economy \rm is a 4-tuple\acr
$ {\mathcal E}=\{ (X,\mathcal{A},\mu);~ \mathbb{R}_+^n;~ e;~\{\succ_a\}_{a\in X} \},$ where:
\begin{itemize}
\item[\bf (E1)]  ({\sl Perfect competition})   the \sl space of agents \rm is a triple $(X, \mathcal{A}, \mu),$ where $(X, \mathcal{A})$ is
a measurable space and $\mu$ is a fuzzy semiconvex submodular such that the  ideal of $\mu$-zero sets is stable under countable unions and satisfying ({\bf H.0}). Each set $E_i$ denotes the set of agents of type $i$.
\item[\bf (E2)] The finite dimensional space $\erre ^n$ is the \sl commodity space\rm, its positive cone  $\mathbb{R}_+^n$ represents the {\em  consumption set} of each agent.
\item[\bf (E3)] Each consumer $a \in E_i$ is characterized by its initial endowment $e(a) = e_i$.
Since  the initial endowment density $e:X \rightarrow (\mathbb{R}_+^n)^{\circ}$ is simple and  constant on the sets $E_i$, $i=1,\ldots, r$ its aggregate initial endowment is
$\mu_e(X) \in (\mathbb{R}_+^n)^{\circ}$. 
Moreover $\mu_e$ is a fuzzy submodular.
\item[\bf (E4)] $\{\succ_i\}_{i \leq r}$ is the {\sl preference relation} associated to the agents
$a\in E_i$, (namely  for every $a \in E_i$ $x \succeq_a y$ means
$x \succeq_i y$ and this is interpreted as "the boundle $x$ is at least as good as the boundle 
$y$ for the consumer $a\in E_i$").
The preference relation is:
\begin{itemize}
\item[(a)] irreflexive and transitive;
 \item[(b)] ({\sl Monotone}) for every $x\in \mathbb{R}_+^n$ and every $v\in  \mathbb{R}_+^n \setminus\{0\}$, 
$x+v\succ_i x$ for all  $i \leq r$;
 \item[(c)] ({\em continuous})  for all $x \in (\mathbb{R}_+^n)^{\circ}$ the set 
$\{y \in \mathbb{R}_+^n: y \succeq_i x \}$ is closed in $\mathbb{R}_+^n$ for all $i \leq r$.
\end{itemize}
 In other words, in each coalition $E_i$, agents share both the same initial endowment and the same preference
criterion.
\end{itemize}
\end{as}
 The  condition  ({\bf H.0}) for example models an economy with $r$ agents as a continuum of economies where
 the $i$-th agent is the representative of infinitely many {\em  identical} agents. 
Moreover  this model can be considered as  representative of an economy with $r$ non-homogeneus agents, where the relative
 influence of the $i$-th agent is given by the measure $\mu(E_i)$ for every $i \leq r$.\acr

\noindent
The following classical concepts of equilibrium theory are recalled in this  setting:
\begin{itemize}
\item an {\em  allocation}  is an
integrable function $f:X \longrightarrow \mathbb{R}^n_+$; an allocation is {\em  feasible} if
$\mu_f(X) = \mu_e(X)$;
\item a \sl price\rm\
is any element $p\in \mathbb{R}_+^n \setminus \{0\}$;
\item the {\em  budget set}  of an agent $a$
of type $i$
for the price $p$ is the set 
$B_{p}(a)=\{x\in \mathbb{R}_+^n:\quad p x\le p e_i\};$
\item a {\em  coalition} is a measurable subset $S$ of $X $ such that  $\mu(S)>0$.
\item A coalition $S$ can \it improve \rm\ the
allocation $f$ if there exists an allocation $g$ such that
\begin{itemize}
\item[(1)] $g(a)\succ_a f(a) \quad \mu$-a.e. in $S$;
\item[(2)] $\mu_g(S) = \mu_e(S)$.
\end{itemize}
\item A coalition $S$ {\em  strongly  improves} the
allocation $f$ if there exists an allocation $g$ such that
\begin{itemize}
\item[(1)] $g(a)\succ_a f(a)\quad  \mu$-a.e. in $S$;
\item[($2^{\prime}$)] $\mu_g (S\cap E_i) = \mu_e (S\cap E_i)$
for all $i=1,...,r$.
\end{itemize}
\item \rm The \it core \rm  $C({\mathcal E})$ of an economy ${\mathcal E}$ is the set
of all the feasible allocations that cannot be improved by any coalition.
\item \rm The \it large core \rm $LC({\mathcal E})$ of an economy ${\mathcal E}$ is the set
of all the feasible allocations that cannot be strongly improved by any coalition.
It is clear that $f\notin \LCEs$ implies $f \notin \CEs$, so $\CEs\subset \LCEs$.
\item\rm\ A {\em  Walras equilibrium }
of ${\mathcal E}$ is a pair
$(f,p)\in L^1_C(\mu, \mathbb{R}_+^n)\times(\mathbb{R}_+^n \setminus \{0\})$
such that:
\begin{itemize}
\item[(i)] $f$ is a feasible allocation;
\item[(ii)] $f(a)$ is a maximal element of $\succ_a$ in the
budget set $B_{p}(a)$, 
(namely $f(a) \in B_p (a)$ and $x \succ_a f(a)$ implies $p \cdot x > p \cdot e(a)$)
for $\mu$-almost all $a\in X$.
\end{itemize}
\item A {\it walrasian allocation} is a feasible allocation $f$ such
that there exists a price $p$ so that the pair $(f,p)$ is a Walras
equilibrium.
\item $W({\mathcal E})$ is the set of all the walrasian allocations of ${\mathcal E}$.
\end{itemize}

The aim of this research is to obtain relations between  Walras equilibria  ${\mathcal W}(\mathcal{E})$ 
and core of an economy ${\mathcal C}({\mathcal E})$.
\noindent
In order to study relations between ${\mathcal C}({\mathcal E})$ and  ${\mathcal W}(\mathcal{E})$,  
the following inclusion is proved. 
\begin{theorem}\label{WESCES}
Under assumptions \rm ({\bf E1}) -- ({\bf E4}),
\em there holds $\CEs \supset \WEs$.
\end{theorem}
\begin{proof}
Let $f \in \WEs \setminus \CEs$.
Then there exist a coalition  $S$ and a feasible allocation $g$ 
 such that $\mu$-a.e. in $S$
$g(a) \succ_a f(a)$~ and ~ 
$\mu_g (S) = \mu_e (S)$.
On the other side there exists a price  $p$ for which $f(a)$ is $\succ _a$ maximal in  $B_p (a)$ $\mu$-a.e.
in $S$. \par \noindent
Consequently, setting $S_1 = \{ a \in S : p \cdot g(a) \leq p \cdot e(a) \},$ it should be $\mu(S_1) = 0$,
otherwise 
the elements $\{x=g(a):a\in S_1\}$ would contradict maximality of $f$, since $g(a)\succ_a f(a)$ but $p\cdot g(a)\leq p\cdot f(a)$
 for all $a\in S_1.$
Hence $\mu$-a.e. in $S$ one has:
\begin{eqnarray*}
p \cdot g(a) &=& \sum_{i=1}^n p_i g_i(a) > \sum_{i=1}^n p_i e_i (a) = p \cdot e(a)
\end{eqnarray*}
whence, by \ref{corollario}.b) and applying Lemma \ref{eq-zero} with $p\cdot g$ and $p\cdot e$ in place of
 $f$ and $g$ respectively,
\begin{eqnarray*}
\iS p \cdot g(a) d\mu&=& \iS \sum_{i=1}^n p_i g_i(a) d\mu>
\iS \sum_{i=1}^n p_i e_i (a) d\mu=
\sum_{i=1}^n p_i \iS e_i (a) d\mu= \\ &=& 
p \cdot \iS e d\mu.
\end{eqnarray*}
Thus
\begin{eqnarray}\label{uno}
\iS p \cdot g d\mu> p \cdot \iS e d\mu.
\end{eqnarray}
On the other side, since $g$ improves $f$, it is:
\begin{eqnarray*}
\iS p \cdot g d\mu\leq p \cdot \iS g(a) d\mu= p \cdot \iS e d\mu
\end{eqnarray*}
and this contradicts (\ref{uno}).
\end{proof}

Of course, Theorem \ref{WESCES} proves also that $\WEs\subset \LCEs$.

\begin{as}\label{gamma} \rm
 Suppose now that the allocation $f$ 
has a representation of the following type:
$$f(a) = \sommar a_i 1_{E_i}$$
 and consider the multifunction
$$\Gamma_f(a) := \{x \in \mathbb{R}_+^n: x \succeq_a f(a)\} =\sommar C_i 1_{E_i},$$
 where the $C_i$ are convex, closed and contain the sets $y + (\mathbb{R}_+^n)^{\circ}$ when $y \in C_i$. 
The class of its Choquet integrable selections is
\[S^*_{\Gamma_f} = \{ \psi \in L^1_{C} (\mu, \mathbb{R}_+^n) \text{ with } \psi (a) 
\in \Gamma_f (a) \text{ for } \mu - \text{ a.e. } a\in X\}.\]
\end{as}

\begin{osservazione}\label{gamma1} \rm
Since $\Gamma_f$ is simple  it contains as selections all functions 
that are $\mu$-a.e. constant in $E_i$ (the constant must be an element of $C_i$).
So all functions of the type $\sum_{i=1}^r c_i 1_{E_i}, c_i \in C_i$ 
are Choquet integrable selections of $\Gamma_f$.
For every $E$ in $\mathcal{A} $ let
$$M_{\Gamma_f} (E) = \left\{
\mu_{\psi} (E),~~ \psi \in S^*_{\Gamma_f} \right\}$$
and  consider its {\it range} 
$R(M_{\Gamma_f} ) = \bigcup _{E\in \mathcal{A}} M_{\Gamma_f} (E).$\acr
\par \noindent
\end{osservazione}

Let 
\begin{eqnarray}\label{I}
I_f := \left\{ z = \mu_s (A) - \mu_e (A), \quad \forall \quad A \in \mathcal{A},
\quad s \in S^*_{\Gamma_f} \right\}.
\end{eqnarray}

Now, in order to prove the convexity of $I$  some preliminary results are needed;
the first is a density result of the multivalued integral of $\Gamma_f$.\acr

\begin{lemma}\label{gsemplice}
If $s \in S^*_{\Gamma_f}$ then, for every $A \in \mathcal{A}$, there exists a "simple" 
selection $g \in S^*_{\Gamma_f}$ such that
$ \mu_s (A) = \mu_g (A)$.
\end{lemma}
\begin{proof}
Let $A \in \mathcal{A}$ and $s = (s_1, s_2, \ldots, s_n)  \in S^*_{\Gamma_f}$ be fixed.
First of all  observe that,
for every $i \in J = \{i \leq r : \mu(A \cap E_i) > 0\}$,  the vectors
\begin{eqnarray}\label{teorema-media}
w^{(i)} := \dfrac{\int^*_{A \cap E_i} s d\mu}{\mu (A \cap E_i)} \in C_i.
\end{eqnarray}
This is the mean value theorem for the Choquet integral and it is a consequence of 
the Hahn-Banach theorem and ({ \ref{De}.v}), 
as in the countably additive case.\acr
Suppose in fact, by contradiction, that there exists $j\in J$ such that
$w^{(j)} \not\in C_j$, then by the Hahn-Banach theorem there exist a positive 
 functional $p$ and a positive  number $a$ such that
\[ \sum_{i=1}^n p_i w^{(j)}_i < a \leq \sum_{i=1}^n p_i y_i, ~ \forall (y_i)_i \in  C_j.\]
Then, in particular, for all $x \in E_j$,
\begin{eqnarray*}
\sum_{i=1}^n p_i w^{(j)}_i &=& \sum_{i=1}^n p_i \dfrac{\mintAEi s_i d\mu}{\mu (A \cap E_j)} < a
 \leq \sum_{i=1}^n p_i s_i(x)
\end{eqnarray*}
and  integrating on $A \cap E_j$ it follows that
\begin{eqnarray*}
\int^*_{A \cap E_j} \sum_{i=1}^n p_i w^{(j)}_i  d\mu &<& a \mu(A \cap E_j) \leq
 \int^*_{A \cap E_j} \sum_{i=1}^n p_i s_i(x) d\mu \leq \\
&\leq& \sum_{i=1}^n p_i \int^*_{A \cap E_j} s_i d\mu =\sum_{i=1}^n p_iw_i^{(j)}\mu(A\cap E_j) < 
a \mu(A \cap E_j)
\end{eqnarray*}
which is clearly absurd. \acr
Let
$E= \cup_{j \in J} E_j$ and
$g(x) = \sum_{i \in J} w^{(i)} \uno_{E_i} + s(x) \uno_{X \setminus E}$.
Then $g \in S^*_{\Gamma_f}$ because it is a "sum" of selections.
Now the following equality
\[ \int^*_A s d\mu= \int^*_{A \cap E} g d\mu= \int^*_A g d\mu\]
 will be proven  component by component. Let $k \leq n$:
\begin{eqnarray}\label{ints-intg}
\int^*_A s_k d\mu&=& \nonumber \int_{A \cap E}^* s_k d\mu =
\int_0^{\infty} \mu (\{x \in A \cap (\cup_{i \in J} E_i : s_k(x) > t\}) dt = \\  &=& \nonumber
\int_0^{\infty} \sum_{i \in J} \mu (\{x \in A \cap E_i : s_k (x) > t\}) dt =
\\ &=& \nonumber
\sum_{i \in J} \int_0^{\infty}  \mu (\{x \in A \cap E_i : s_k(x) > t\}) dt =\\ &=& 
\sum_{i \in J} \int^*_{A \cap E_i} s_k d\mu
= \sum_{i \in J} w^{(k)}_i \mu (A \cap E_i) = \int^*_{A \cap E} g_k d\mu.
\end{eqnarray}
Then the vector inequality
\[ \int^*_A s d\mu= \int^*_{A \cap E} g d\mu \quad \underline{\ll } \mintA  g d\mu\]
is proven thanks to the monotonicity of $\mu_g$ and {\bf (H.0)}.
Finally, since $\mu(A \setminus E) = 0$, 
 it follows easily
$$ \mintA  g d\mu =  \int^*_{A \cap E} g d\mu.$$
 \end{proof}

\begin{osservazione}\label{iperi}\rm
Observe that, in the previous Lemma, it has been proved that, as soon as $s \in S^*_{\Gamma_f}$,
 for every fixed $i\leq r$ and every measurable $A\subset E_i$ an element $w^{(i)}\in C_i$ can be found, such that
$$\mintA  s d\mu=w^{(i)}\mu(A).$$
\end{osservazione}

\begin{proposizione}\label{slsemiconvex}
Let $\mu: \mathcal{A} \rightarrow \mathbb{R}^0_+$ be a semiconvex fuzzy measure with property {\rm \bf(H.0)}.
Then the range $R(\mu, \mu_e)$ is convex.
\end{proposizione}
\begin{proof} Let  $R_i$ be the range of the $(n+1)$-dimensional fuzzy measure 
$(\mu, \mu_e)$, when restricted to $E_i$, $i=1,...,r$ and let $i\leq r$. Since $e$ is constant in $E_i$, it is
$$(\mu,\mu_e (A))=\mintA (1,e)d\mu=(\mu(A),\mu(A)e)$$ for every $A\subset E_i$. Since the range of $\mu$, when restricted to $E_i$, is the interval $[0,\mu(E_i)]$, then $R_i$ is the segment joining the origin with the point $(\mu(E_i),\mu(E_i)e)$. From the property {\bf (H.0)} and Proposition \ref{intadditivo}, it follows easily that $R=\sum_{i=1}^r R_i$, hence $R$ is convex.
\end{proof}

\begin{osservazione}\label{lambdaconvessa}\rm
Clearly, from Proposition \ref{slsemiconvex} it follows also that the range of $\mu_e$ ($R(\mu_e )$) is convex. In particular, given $A$ and $B$ in ${\mathcal A}$, and fixed $t\in [0,1]$, for each $i=1,...,r$ there exists a measurable set $D_t^i\subset E_i$ such that 
$$\mu_e(D_t^i)=t\mu_e (A\cap E_i)+(1-t)\mu_e (B\cap E_i).$$
Then clearly the set $D_t=\bigcup_iD_t^i$ satisfies
$\mu_e (D_t)=t\mu_e (A)+(1-t)\mu_e (B).$\\
\end{osservazione}

Using Lemma \ref{gsemplice} and Proposition \ref{slsemiconvex} the following result holds
\begin{theorem}\label{1.3.1.cond-c}
The set ${I_f}$ given in formula {\rm(\ref{I}\rm)} is convex.
\end{theorem}
\begin{proof}
Thanks to Proposition \ref{intadditivo}  it is easy to see that
$I_f=\sommar I_i,$
where 
$$I_i:=\{z= \mu_s(A) -\mu_e (A), A\in \mathcal{A} \cap E_i, s\in S^*_{\Gamma_f}\}.$$
Then, it will suffice to prove that each $I_i$ is convex. 
In particular it will be proven that
$I_i=[0,\mu(E_i)]  (C_i-e_i).$ 
Since $C_i$ is convex, so is $C_i-e_i$ and also the cone $[0,\mu(E_i)](C_i-e_i)$.
So, fix any element $z\in I_i$: then there exist $s\in S^*_{\Gamma_f}$ and a measurable $A\subset E_i$ such that
$$z=  \mu_s(A)-\mu_e (A)=  \mu_s(A) -e_i\mu(A).$$
Thanks to the Remark (\ref{iperi}), there exists an element $w\in C_i$ such that $\int_A sd\mu=w\mu(A)$, hence
$$z=(w-e_i)\mu(A)\in (C_i-e_i)\mu(A)\subset(C_i-e_i)[0,\mu(E_i)].$$
Conversely, for every $w\in C_i$, there exists a selection $s$ such that $s_{|E_i}$ is constantly equal to $w$. Moreover, for any real number $x\in[0,\mu(E_i)]$ there exists a measurable set $A\subset E_i$ such that $\mu(A)=x$ and $\mu_e (A)=e_ix$. Therefore
$$(w-e_i)x=  \mu_s(A) -\mu_e (A)\in I_i.$$
From arbitrariness of $w$ and $x$, it follows the converse inclusion, $(C_i-e_i)[0,\mu(E_i)]\subset I_i$. \end{proof}

\begin{lemma}\label{separazione}
If $f\in \LCEs$, there exists $p \in \mathbb{R}_+^n$, $p \neq 0$ such that $p \cdot x \geq 0$ for all $x \in \overline{I}_f$.
\end{lemma}
\begin{proof}
First
$I_f \cap (-\mathbb{R}_+^n) = \{ 0\}$ will be proven.
Indeed, assume by contradiction that there exists
$z\in I_f \cap (-\mathbb{R}_+^n)$ with $z\not= 0$. Then there exist a coalition $A\in \mathcal{A}$ and
a Choquet integrable selection $s\in S^1_{\Gamma_f}$ such that
$$z =  \mu_s(A) -  \mu_e(A) \in (-\mathbb{R}_+^n).$$
Then it follows 
immediately that $\mu(A) > 0$ (otherwise both  $\mu_s(A) = 0$
and $ \mu_e(A) = 0$
whence $z = 0$).\acr
Observe that $z=\sum_{i=1}^r z_i$, where $z_i= \mu_s(A \cap E_i) - e^i\mu(A\cap E_i)$. 
Now, let $J:=\{i: z_i\neq 0\}$. Of course, $J\neq \emptyset$ otherwise $z=0$, again.
Clearly, $z=\sum_{j\in J}z_j$. 
Now, for each $j\in J$, it is $\mu(A\cap E_j)>0$ (otherwise $z_j=0$), and  define 
$s^j:=s \uno_{E_j}$. 
Moreover, define $A':=\bigcup_{j\in J}(A\cap E_j)$, and finally let us set
$$s_0:=\sum_{j\in J} (s^j-\dfrac{z_j}{\mu(A\cap E_j)})\uno_{E_j}.$$
Since $z_j\in (-\erre^n_+)\setminus \{0\}$ for each $j$, then the allocation $s_0$ satisfies 
$s_0(a)\succ_a s(a)\succ_af(a)$ $\mu$-a.e. in $A'$, and moreover 
$$\int_{A'\cap E_j}^* s_0 d\mu=\int_{A'\cap E_j}^* s^jd\mu-z_j=e^j\mu(A\cap E_j)$$
holds true, for all $j\in J$. 
Moreover, if $i\notin J$,  by definition it is
$A'\cap E_i=\emptyset$, and so $$\int_{A'\cap E_i}^* s_0 d\mu=0=\int_{A'\cap E_i}^* e d\mu.$$
So it is proved that the coalition $A'$ strongly improves $f$ by the allocation $s_0$.
 But this is impossible, since $f\in \LCEs$.\acr
In conclusion $I_f \cap (-\mathbb{R}_+^n) = \{ 0\}$ and hence
 $\overline{I}_f \cap (-\mathbb{R}_+^n)^o = \emptyset$.
Since both sets are convex, and the second one has non-empty interior, we can  apply the 
Strong Separation Theorem, and
determine some $p \in \erre ^n$ $p\not= 0$ such that $p \cdot x \geq 0$ for all $x \in \overline{I}_f$.\acr
It only remains to prove that
$p\in \mathbb{R}_+^n.$
Indeed, we have that $(\mathbb{R}_+^n)^o \subset I_f;$ in fact if $x\in (\mathbb{R}_+^n)^o$, then 
the allocation $\displaystyle{\psi =\dfrac{x}{\mu (X)} + f}$ is in $S^*_{\Gamma_f}$ and
\begin{eqnarray*}
\mintX \psi d\mu- \mintX e d\mu= \mintX fd\mu+ x - \mintX e d\mu= x
\end{eqnarray*}
since $f$ is feasible. Then  
$p\cdot x \geq 0$ for every $x\in (\mathbb{R}_+^n)^o,$ whence necessarily
$p\in \mathbb{R}_+^n$. \end{proof}

\noindent

Under assumption ({\bf \ref{gamma}}) it follows that
\begin{proposizione}\label{strassen} 
Let $p$ be as in Lemma  { \rm \ref{separazione}} and 
$\gamma$ be the map defined by:
$\gamma(a) = \inf \{ p\cdot y :  y \in (\Gamma_f(a) - e(a))\cup \{0\} \quad \}.$
Then $\gamma$ is identically null.
\end{proposizione}
\begin{proof}
 Observe that, since $\Gamma_f$ and $e$ are simple, $ (\Gamma_f - e)\cup \{0\} $ is graph-measurable, that is
\begin{eqnarray*}
\{ (a,x)\in X \times \erre^n_+: x\in  (\Gamma_f(a) - e(a))\cup \{0\}  \} \in \mathcal{A} \otimes {\mathcal B}_{\erre^n}.
\end{eqnarray*}
Since $f$ is a simple function then, for every 
$a \in E_i$, 
$ \gamma (a) = 0 \wedge [ \inf_{y \in C_i} p \cdot y - p \cdot e_i] := \gamma_i.$\acr
Let $I_0$ be the set $I_0 = \{ i \leq r : \gamma_i < 0\}$. It will be proven that $I_0$ is empty.
Suppose by contradiction that there exists $i \in I_0$, namely $\gamma_i = \inf_{y \in C_i} p \cdot y - p \cdot e_i < 0$.
Since $C_i$ is closed there exists $x_i \in C_i$ such that $\gamma_i = p \cdot x_i - p \cdot e_i$.
Let $s$ be the function defined by: $s= x_i \uno_{E_i} + f \uno_{X \setminus E_i}$.
Observe that
$$ \int_{E_i} s d\mu- \int_{E_i} e d\mu =  (x_i  - e_i)\mu(E_i) \in I_f.$$
So, by Lemma \ref{separazione},
$p \cdot (x_i - e_i) \mu(E_i) = \mu(E_i) p \cdot (x_i - e_i) \geq 0$, while, by hypothesis,
$\mu(E_i) > 0$ and $p (x_i - e_i) = \gamma_i < 0$: contradiction.
\end{proof}

\begin{theorem}\label{finale}
Under Assumptions  
\rm ({\bf E1}) -- ({\bf E4}) and  {\rm \ref{gamma})}  if $f \in \LCEs$ is a
 {\em simple} allocation, then $f \in \WEs$.
\end{theorem}
\begin{proof}
Thanks to Theorem \ref{WESCES} it is $\WEs\subset \CEs$.
To prove the converse inclusion fix
 $f \in \LCEs$. Consider $ (\Gamma_f - e) \cup \{0\}$ 
and  let $I_f$ be its Aumann integral obtained via
Choquet integrable selections. By  Lemma \ref{separazione},
it is known that
$I_f \cap (-\mathbb{R}_+^n) = \{0\}$ 
and $p \in \mathbb{R}_+^n$ exists such that
\begin{eqnarray}\label{px}
px \geq 0 \quad \text{ for every } x \in I_f.
\end{eqnarray}
By Proposition \ref{strassen}
a.e. in 
$X, p\cdot e(x)\leq p\cdot y$, for every $y\succeq_{x} f(x)$.
So, by continuity of the preorder,
a.e. in $X, p \cdot e(x) \leq p \cdot f(x)$.
It will be proven now that the previous inequality is in fact
an equality.  Let $b_k = p \cdot e(x)$, for every $x \in
E_k$; by Corollary \ref{corollario}
\begin{eqnarray*}
\mintX p \cdot e(x) d\mu= p \cdot \mintX
e(x) d\mu= \sum_{k=1}^r b_k \mu(E_k).
\end{eqnarray*}
Let $A \in \mathcal{A}$ be fixed. Then
\begin{eqnarray}\label{sommafinale}
\nonumber 
 0 &\leq& \mintA  \hskip-.2cm  p( f - e) d \mu =
\sum_{k=1}^r \int_{A\cap E_k}^*  \hskip-.6cm p(f - e) d\mu= 
 \sum_{k=1}^r \int_0^{\infty} \hskip-.4cm \mu
(\{x \in A \cap E_k: p  f(x) > b_k + t \}) dt = \\
\nonumber &=& 
\sum_{k=1}^r \int_{b_k}^{\infty}\mu (\{x \in A \cap E_k: p f(x) > u\}) du = \\
\nonumber &=&
\sum_{k=1}^r \left( \int_0^{\infty} \mkern-20mu \mu (\{x \in A \cap E_k: p  f(x) > u\}) du  -
\int_0^{b_k} \mkern-20mu \mu (\{x \in A \cap E_k: p  f(x) > u\}) du \right) \leq
\\
&\leq& 
\sum_{k=1}^r \left(\int_{A \cap E_k}^* p  f d\mu- \int_0^{b_k}  \mu (\{x \in A \cap E_k: p  e(x) > u\}) du \right) = \\
\nonumber &=& 
\sum_{k=1}^r \left( \int_{A \cap E_k}^* p \cdot f d\mu- 
\int_{A \cap E_k}^*
p\cdot e d\mu \right) 
= \mintA p \cdot f d\mu- \mintA p \cdot e d\mu= \\ &=& \nonumber
p \left(\mintA f d\mu- \mintA e
d\mu\right).
\end{eqnarray}
For $A = X$, since $f$ is feasible and $p\cdot f -p\cdot e \geq 0$ $\mu$-a.e., it follows 
\[ 0 
\leq \mintX  p( f - e) d \mu =
p \left(\mintX f d\mu- \mintX e d\mu\right) = 0;\] 
 this in
turn implies that
$\mu_{pf} = \mu_{pe}$ on $\mathcal{A}$.
Applying Corollary \ref{ugualeqo}
we  get $p\cdot f = p\cdot e$
$\mu$-a.e. in $X$.\acr
The remaining part of the proof is exactly the same as that of
\cite[Theorem 2.1.1, pag 133 ff]{Hilde}  
since preferences are assumed to be monotone and continuous.  
\end{proof} 
\medskip

\subsection{Existence of Equilibria}
Assume now that the preferences have the following structure:
\begin{itemize}
\item[$\bullet$] there exist $r$ subsets $J_1, \ldots, J_r \subseteq \{1, 2, \ldots, n\}$ such that:
\begin{itemize}
\item[i)] for every $x, y \in C$, $x \succ_a y \Longleftrightarrow x_j > y_j, j \in J_k$ when $a \in E_k$;
\item[ii)] $\bigcap_{i=1}^r J_i \neq \emptyset$.
\end{itemize}
\end{itemize}
This means that within each coalition $E_k$ only the items of the $k$-th list $J_k$ are considered, 
in order to decide whether a bundle is preferred to another. Observe that such assumption does not 
fulfil monotonicity, in the sense of ({\bf A.3.b}), but it satisfies the more demanding form
\begin{itemize}
\item for every $x \in \mathbb{R}_+^n$,  $z \in (\mathbb{R}_+^n)^0$, then $x + z \succ_a x$ for every $a \in X$.
\end{itemize}
However Lemma \ref{separazione} remains true: one has only to note that $I_f \cap (-\mathbb{R}_+^n)^0 = \emptyset$ 
with the same proof.

\begin{proposizione}\label{cor-finale}
Under Assumptions  \rm ({\bf E1}) -- ({\bf E4}), $e  \in \WEs$.
\end{proposizione}
\begin{proof}
It is enough to prove that $e  \in \CEs$ and then apply Theorem \ref{finale}. Assume by contradiction that
$e  \not\in \CEs$; then there exists a pair $(f,S)$ that improves $e$, namely
\begin{itemize}
\item[\ref{cor-finale}.a)] $f \succ_a e$, when $a \in S$;
\item[\ref{cor-finale}.b)] $\iS f d\mu= \iS e d\mu$.
\end{itemize}
From \ref{cor-finale}.a), if $k \in \bigcap_{i=1}^r J_i$, we have for the $k$-th entries of $f$ and $e$,
$f_k(a) > e_k(a), a \in S$. Hence by Corollary \ref{ugualeqo}, there holds
$$ \iS f_k d\mu> \iS e_k d\mu$$
that contradicts \ref{cor-finale}.b). 
\end{proof}

It will be shown now that, in some cases, allocations that are constant in the sets $E_i$ turn out to be important when searching 
elements of the core $\CEs$.\acr

A technical result will be established first, concerning the Choquet integral in this context. 

\begin{lemma}\label{cmaggiore}
Let $\mu$ be  submodular, $f$ be any scalar integrable function on $X$ and $c$ any positive real constant. If there exists a set 
$S\in \mathcal{A}$, with positive measure, such that $\mintS  cd\mu\geq \mintS  f d\mu$, then either
 $f\equiv c$ $\mu$-a.e. in $S$ or there exists a measurable subset $S'\subset S$, with $\mu(S')>0$, such that 
$f(s)<c$ for all $s\in S'$.
\end{lemma}
\begin{proof} 
Assume directly that $f$ and $c$ are not $\mu$-a.e. equal in $S$, and  define:
$$H:=\{s\in S:f(s)\neq c\},\quad  H_1:=\{s\in H: f(s)<c\}, \ H_2:=\{s\in H: f(s)>c\}.$$
Then $\mu(H)>0$. By contradiction,  suppose that $\mu(H_1)=0$, then $f\geq c$ $\mu$-a.e. in $S$. 
By (\ref{De}.iii) it follows that $\mintS  fd\mu \geq \int_S^* c d\mu$ and so the two integrals coincide.
\acr
Now, consider the non-negative map $f-c$ in $S\setminus H_1$:  thanks to (\ref{De}.iv) it is
$$0=\mintS  f d\mu-c\mu(S)=\int_{S\setminus H_1}^* f d\mu -\int_{S\setminus H_1}^* c d\mu=
\int_{S\setminus H_1}^* (f-c)d\mu$$
But also
\begin{eqnarray*}
0&=&\int_{S\setminus H_1}^* (f-c)d\mu\leq \int_{H_2}^* (f-c) d\mu+
\int_{S\setminus H}^* (f-c)d\mu=\int_{H_2}^* (f-c) d\mu\leq\\
&\leq& \int_{S\setminus H_1}^* (f-c)d\mu.
\end{eqnarray*}
So, 
$\int_{H_2}^* (f-c) d\mu=0$
and $\mu(H_2)$ is positive, otherwise $\mu(H)=0$. Then, from  Lemma \ref{eq-zero}, it would follow 
$\int_{H_2}^*(f-c) d\mu>0$: contradiction. 
\end{proof}

Now,  assume that the preference preorders $\succ_i$ are of a special type, i.e.  suppose that they are
represented by $r$ utility functions $u_i$, such that 
\begin{itemize}
\item[\bf u1)] each $u_i:\erre^n_+\to \erre$ is positive, increasing, continuous, and concave;
\item[\bf u2)]  each $u_i$ is positively homogeneous in the one-dimensional subspace generated by the vector 
$e^i:=e \uno_{E_i}$.
\end{itemize}

Observe that a Jensen-type theorem holds true also
in this case, e.g.
$$\mintX u(f_1,...,f_k)d\mu\leq u(\mintX f_1 d\mu, ...,\mintX f_k d\mu)$$
as soon as the integrals exist,
 $u$ is concave, increasing, continuous, and $\mu(X)=1$.

\begin{rem}\rm
Observe that in the paper \cite{WANG} the dual inequality is stated, for a {\em convex} function $u$, but under the condition of 
comonotonicity of the mappings $f_1, \ldots,f_n$.
\end{rem}

Following \cite{WANG},  the following notion of concave function in the positive orthant of $\erre^n$ is adopted:
\begin{definition}\label{concavendim}\rm
An increasing map $u:\erre^n_+\to \erre^+_0$ is {\em concave} if, for every $n$-tuple $(t_1,...,t_n)\in \erre^n_+$ there exist $n$  
non-negative parameters $a_1,...,a_n$ and a real constant $c$, such that
\begin{itemize}
\item[\bf(\ref{concavendim}.1)] \hskip.5cm
$a_1t_1+...+a_nt_n+c=u(t_1,...,t_n) $\quad
and
\item[\bf(\ref{concavendim}.2)] \hskip.5cm
$a_1\tau_1+...+a_n\tau_n+c\geq u(\tau_1,...,\tau_n)$\quad
for all $(\tau_1,...,\tau_n)\in \erre^n_+$.
\end{itemize}
\end{definition}
\medskip

Now the following result will be proven:
\begin{theorem}\label{jensen2}
If $u$ is an increasing and concave function according to  the previous definition, and $f_1,...,f_n$ are scalarly integrable allocations such that $u(f_1,...,f_n)$ is integrable too, then, provided that $\mu(X)=1$:
$$\mintX u(f_1,...,f_k)d\mu\leq u(\mintX f_1 d\mu, ...,\mintX f_k d\mu).$$

\end{theorem}
\begin{proof}
 Set $t_i:=\mintX f_id\mu$ for all $i$, and let $a_i$, $c$ be the constants given in the previous definition.
So, thanks to the properties of the asymmetric integral it is:
\begin{eqnarray*}
u(\mintX f_1 d\mu, ...,\mintX f_k d\mu) &=&u(t_1,...,t_n)=a_1t_1+...+a_nt_n+c=  \\ &=&
 a_1\mintX f_1d\mu+...+a_n\mintX f_nd\mu+\mintX cd\mu\geq \\
&\geq& \mintX (a_1f_1+...+a_nf_n)d\mu +\mintX cd\mu= \\
&=& \mintX (a_1f_1+...+a_nf_n+c)d\mu\geq \mintX u(f_1,...,f_n)d\mu,
\end{eqnarray*}
the last inequality following from (\ref{concavendim}.2). 
\end{proof}

Before stating the next results,  some notations will be introduced. Given any integrable allocation
 $f$, let $\overline{f}$ be the following map:
$$\overline{f}=\sum_{i=1}^r\dfrac{\mintEi   fd\mu}{\mu(E_i)} 1_{E_i}.$$ 
Clearly, $\overline{f}$ is constant in each set $E_i$, and we call it the {\em average function} of $f$. 
Of course, $f$ is feasible if and only if $\overline{f}$ is.
The next result states that, when preferences are of the type described above, given a feasible allocation 
$f\in \LCEs$, then also its average function $\overline{f}$ belongs to $\LCEs$.
This fact, thanks to Theorems \ref{WESCES} and \ref{finale}, allows  to deduce  that $f\in \CEs$  
implies $\overline{f}\in \CEs\cap \WEs$.

\begin{theorem}\label{ultimo}
In the situation described
above, assume that $f$ is any feasible allocation belonging to $\CEs$. Then $\overline{f}\in \CEs$.
\end{theorem}
\begin{proof} 
As  already observed, it will suffice to show that $\overline{f}\in \LCEs$. Assume by contradiction that there 
exist a coalition $A$ with positive measure and an allocation $g$ such that $g(a)\succ_a \overline{f}(a)$ for
 all $a\in A$ and moreover $\mintAEi gd\mu =\mintAEi e d\mu$ for all $i=1,...,r$. 
Now, fix any index $i$ such that $\mu(A\cap E_i)>0$, and denote by $w^i$ the constant value of 
$\overline{f}$ in $E_i$: we have $u_i(g(a))>u_i(w^i)$ for every $a\in A\cap E_i$ and
$$\mintAEi u_i(g) d\mu>u_i(w_i) \mu(A\cap E_i)\qquad {\rm and } \qquad \mintAEi  gd\mu=
e^i\mu(A\cap E_i).$$
Now, thanks to the properties {\bf u1}) and {\bf u2}), it is
\begin{eqnarray*}
u_i(e^i) &=& u_i\left(\dfrac{1}{\mu(A\cap E_i)}\mintAEi e^i d\mu\right)=u_i\left(\dfrac{1}{\mu(A\cap E_i)}
\mintAEi g d\mu\right)\geq \\
&\geq& \dfrac{1}{\mu(A\cap E_i)}\mintAEi  u(g) d\mu>
\dfrac{1}{\mu(A\cap E_i)}u_i(w_i)\mu(A\cap E_i)=u_i(w^i).
\end{eqnarray*}
This clearly means that $e$ improves $\overline{f}$ in $E_i$. The initial endowment $e$  improves also $f$ in 
some coalition $S\subset E_i$: indeed this is obvious if $u_i(f)=u_i(w^i)$ in $E_i$; otherwise, by concavity,
\begin{eqnarray*}
u_i(w^i)=u_i\left(\dfrac{\mintEi   f d\mu}{\mu(E_i)}\right)\geq \frac{1}{\mu(E_i)}\mintEi  u_i(f)d\mu,
\end{eqnarray*}
i.e. $\mintEi  u_i(f)d\mu\leq \mintEi  u_i(w^i)d\mu$. Now, by Lemma \ref{cmaggiore}  there exists a
 measurable subset $S\subset E_i$, such that $u_i(f(s))<u_i(w^i)<u_i(e^i)$ in $S$: this means that $e$
 improves $f$ in $S$, and therefore $f\notin \CEs$; contradiction. 
\end{proof}

Also in this setting it is possible to  prove that $e\in \CEs$.

\begin{theorem}\label{enelcore}
In the situation described above, the initial endowment $e$ is in $\CEs$.
\end{theorem}
\begin{proof} 
Assume that a coalition $A$ and an allocation $g$ exist, strongly improving $e$. Fix any index $i$ such that 
$\mu(E_i\cap A)>0$. Then, for all $a\in A\cap E_i$ we get
\begin{eqnarray*}
u_i(g(a))&>&u_i(e^i)=u_i\left(\dfrac{\mintAEi e^i d\mu}{\mu(E_i\cap A)}\right)=
u_i \left( \dfrac{\mintAEi  g d\mu}{\mu(E_i\cap A)} \right) \geq \\
&\geq& \dfrac{1}{\mu(E_i\cap A)}\mintAEi u_i(g) d\mu.
\end{eqnarray*}
Thanks to Lemma \ref{eq-zero}, integrating we have then
$$\mintAEi u_i(g) d\mu > \mintAEi u_i(g) d\mu$$
which is clearly absurd. So, $e\in \LCEs$, but this also implies that $e\in \WEs$ and $e\in \CEs$. 
\end{proof}

In the last situation, it is possible to describe what are precisely the {\em simple} elements in the core $\CEs$.
\begin{proposizione}\label{piccolocore}
In the setting outlined
above, for any feasible {\em simple} mapping $f:X\to \erre^n_+$,
 i.e. $f=\sum_{i=1}^r w^i 1_{E_i}$, where $w^i$ is constant for each $i$, the following are equivalent:
\begin{itemize}
\item[(\bf \ref{piccolocore}.1)]
$u_i(w^i)\geq u_i(e^i)$ for all $i=1,...,r$ and $u_j(w^j)= u_j(e^j)$ for at least one index $j$.
\item[(\bf \ref{piccolocore}.2)]
$f\in \CEs.$
\end{itemize}

\end{proposizione}
\begin{proof} First, assume that $f\in \CEs$.  
If there exists an index $j$ for which $u_j(w^j)< u_j(e^j)$, then  clearly $e$ improves $f$ in $E_j$, which is impossible. 
Therefore, $u_i(w^i)\geq u_i(e^i)$ for all $i$. However, if $u_i(w^i)> u_i(e^i)$ for all $i$, then $f$ would improve $e$ in 
the whole space, since $f$ is feasible, and this is in contrast with Theorem \ref{enelcore}. So, there exists at least an index $j$ such that $u_j(w^j)=u_j(e^j)$.
\acr
Now,  assume that $u_i(w^i)\geq u_i(e^i)$ for all $i$, and prove that $f\in \LCEs$: since $f$ is {\em simple}, this will imply
 that $f\in \CEs$. By contradiction, if an integrable allocation $g$ strongly improves $f$ in a coalition $S$, then there exists 
an index $i$ such that $\mu(S\cap E_i)>0$, and $u_i(g(s))>u_i(w^i)\geq u_i(e^i)$ for all $s\in S\cap E_i$.\acr
 Moreover, $\mintSEi   gd\mu=\mintSEi   e d\mu$: so, $g$ improves $e$ in $S\cap E_i$, which is impossible thanks to 
Proposition \ref{enelcore}. \acr
\end{proof}

In the sequel,  the following restrictions to the utility functions $u_i$ are imposed: 
 they are all the same concave function $u$, and  $u$ is (positively) linear in the subspace generated by the vectors 
$e^i:= e \uno_{E_i}, i=1,...,r$. In this situation  it is clear that, for all measurable sets $A$, it is:
\begin{eqnarray*}
\mintA  u(e) d\mu&=&\sum_i \mintAEi u(e) d\mu=\sum_i u(e^i)\mu(A\cap E_i)= \\ &=&
u(\sum_i e^i \mu(A\cap E_i))= u(\mintA e d\mu). 
\end{eqnarray*}
Finally, it will be proven that, in the last restrictive hypotheses on the preferences, every allocation in the core is necessarily {\em simple}, and therefore $\CEs=\WEs$.

\begin{theorem}\label{corengrato}
Under  the previous conditions on the preferences, for every integrable allocation $f$ the following are equivalent:
\begin{itemize}
\item[(\ref{corengrato}.1)] $f\in \WEs$;
\item[(\ref{corengrato}.2)] $f\in \CEs$;
\item[(\ref{corengrato}.3)] $f=\overline{f}$ and $u(f)=u(e)\quad \mu$-a.e.
\end{itemize}
 \end{theorem}
\begin{proof}
 It will be  proven that
(\ref{corengrato}.1)$\Rightarrow$ (\ref{corengrato}.2)$\Rightarrow$(\ref{corengrato}.3)$\Rightarrow$ (\ref{corengrato}.1).
Thanks to Theorem \ref{WESCES}, it's clear that (\ref{corengrato}.1)$\Rightarrow$ (\ref{corengrato}.2).
Implication (\ref{corengrato}.2) $\Rightarrow$ (\ref{corengrato}.3) will be proven now.
\acr
\noindent Fix any allocation $f\in \CEs$. Then, thanks to Theorem \ref{ultimo}, it is $\overline{f}\in \CEs$. 
The equality $u(w^i)=u(e^i)$ for all $i$, where as usual $w^i$ denotes the constant value of the average function $\overline{f}$ in $E_i$, will be proven first. 
Since $\overline{f}$ is feasible and $u$ is concave, it is
$$\dfrac{\mintX u(\overline{f})d\mu}{\mu(X)}\leq u\left(\dfrac{\mintX \overline{f}d\mu}{\mu(X)}\right)=u\left(\dfrac{\mintX e d\mu}{\mu(X)}\right)=\dfrac{\mintX u(e) d\mu}{\mu(X)}.$$
So,
$$\sum_iu(w^i)\mu(E_i)\leq\sum_iu(e^i)\mu(E_i).$$
But it is known that $u(w^i)\geq u(e^i)$ for all $i$, thanks to Proposition \ref{piccolocore}, hence the last inequality would be violated
 if $u(w^i)> u(e^i)$ for some index $i$: summarizing, it is $u(e^i)=u(w^i)$ for all $i=1,...,r$.
\acr
Now  $f=\overline{f}$ $\mu$-a.e. will be proven: otherwise,
there exist an index $i$ and a measurable set $A\subset E_i$, such that $\mu(A)>0$ and $f \uno_A(a)\neq w^i$ for all $a\in A$.
 For all $a\in E_i$ it is
$$u(\overline{f}(a))=u\left(\dfrac{\mintEi  fd\mu}{\mu(E_i)}\right)\geq \dfrac{1}{\mu(E_i)}\mintEi  u(f)d\mu,$$
since $u$ is concave. Thus
$$\mintEi   u(f)d\mu\leq \mu(E_i), \quad u(\overline{f})=\mintEi  u(w^i) d\mu.$$
Thanks to Lemma \ref{cmaggiore}  there exists a measurable set $S\subset E_i$ with $\mu(S)>0$ and $u(f(a))<u(w^i)$ for all $a\in S$. But 
$u(w^i)=u(e^i)$, for all $i$, hence
$$u(f(a))<u(e^i)$$
holds true, for all $a\in S$, which shows that the coalition $S$ and the allocation $e$ improve $f$: contradiction, since $f\in \CEs$.
Then it is possible to conclude that $f=\overline{f}$ $\mu$-a.e. and the second implication is proved. 
 Finally, if (\ref{corengrato}.3) holds true, $f=\overline{f}\in \CEs$ thanks to Proposition \ref{piccolocore}, and also $f\in \WEs$ thanks
 to Theorem \ref{finale}. Thus, also the last implication is demonstrated.
\end{proof} 

\end{document}